\documentclass[a4paper,twopage,reqno,11pt]{amsart}

\usepackage{pdfsync}
\usepackage{graphicx}
\usepackage{amsmath}
\usepackage{amssymb}
\usepackage{amsfonts}
\usepackage{amsthm}
\usepackage{amstext}
\usepackage{amsbsy}
\usepackage{amsopn}
\usepackage{amscd}
\usepackage{enumerate}
\usepackage{xcolor}
\usepackage{color}
\usepackage{url}
\usepackage[colorlinks]{hyperref}
\usepackage[hyperpageref]{backref}
\usepackage[numbers,sort&compress]{natbib}
\newcounter{rownumber}
\newcommand{\rownumber}{\stepcounter{rownumber}\arabic{rownumber}}
\newcounter{linenumber}
\newcommand{\linenumber}{\stepcounter{linenumber}\arabic{linenumber}}
\usepackage{tabularx}
\usepackage{multirow}
\usepackage{lipsum}
\usepackage{rotating}
\usepackage{longtable}
\usepackage[font=small]{caption}
\usepackage{float}
\usepackage{lscape}
\usepackage{pdflscape}
\usepackage[]{algorithm2e}
\usepackage{color}
\usepackage[colorlinks]{hyperref}
\hypersetup{
	colorlinks=true,%
	citecolor=blue,%
	filecolor=black,%
	linkcolor=red,%
	urlcolor=magenta
}

\newtheorem{theorem}{Theorem}[section]
\newtheorem{corollary}[theorem]{Corollary}

\newtheorem{proposition}[theorem]{Proposition}

\theoremstyle{definition}

\numberwithin{equation}{section}

\newcommand{\GL}{\mathrm{GL}}

\newcommand{\PSL}{\mathrm{L}}

\newcommand{\A}{\mathrm{A}}

\renewcommand{\S}{\mathrm{S}}

\newcommand{\Q}{\mathrm{Q}}
\newcommand{\QD}{\mathrm{QD}}

\newcommand{\D}{\mathrm{D}}

\newcommand{\nr}{\mathrm{nr}}

\newcommand{\HS}{\mathrm{HS}}

\newcommand{\Aut}{\mathrm{Aut}}

\newcommand{\Out}{\mathrm{Out}}

\newcommand{\M}{\mathrm{M}}

\newcommand{\Dmc}{\mathcal{D}}
\newcommand{\Bmc}{\mathcal{B}}

\newcommand{\Pmc}{\mathcal{P}}
\newcommand{\Cmc}{\mathcal{C}}

\renewcommand{\leq}{\leqslant}
\renewcommand{\geq}{\geqslant}

\newcommand{\imod}[1]{\allowbreak\mkern4mu({\operator@font mod}\,\,#1)}

\newcommand{\Bb}[2]{\mathrm{B}_{#2}(#1)}

\begin{document}
	\title[]{Mathieu groups as flag-transitive  automorphism groups of block designs}
	
	\author[F. Bakhtiari]{Fatemeh Bakhtiari}
	\address{Fatemeh Bakhtiari, Department of Mathematics, Faculty of Science, Bu-Ali Sina University, Hamedan, Iran.}
	\email{f.bakhtyari@sci.basu.ac.ir}
	
	\author[A. Daneshkhah]{Ashraf Daneshkhah}
	\address{Ashraf Daneshkhah, Department of Mathematics, Faculty of Science, Bu-Ali Sina University, Hamedan, Iran.}
	\thanks{Corresponding author: Ashraf Daneshkhah}
	\email{adanesh@basu.ac.ir and  daneshkhah.ashraf@gmail.com}

	\subjclass{05B05, 05B25, 20B25, 20D08}%
	\keywords{$2$-design, flag-transitive, sporadic simple group, Mathieu group}
	\date{\today}%

	\begin{abstract}
		In this article, we study $2$-$(v,k,\lambda)$ designs $\Dmc$ admitting a flag-transitive almost simple automorphism group $G$ with socle one of the Mathieu groups. In conclusion, we obtain all such $2$-designs with explicit constructions with appropriate references to known designs. To our knowledge, we obtain $101$ newly-constructed flag-transitive $2$-designs. 
	\end{abstract}
	
	\maketitle
	
	\section{Introduction}\label{sec:intro}
	
	The sporadic simple groups as one of the classes of finite simple groups have always been thought of as an important tool for studying the symmetries of incidence structures, in particular designs, codes and graphs.  Indeed, several sporadic groups arise naturally as automorphism groups of block designs. Among these, the Mathieu groups discovered in the nineteenth century have exceptional symmetric properties. They are in fact unique among finite permutation groups as the only $s$-transitive groups with $s \geq 4$ apart from the symmetric and alternating groups. The Mathieu groups  are intrinsically connected to symmetries of the Witt designs (as the Steiner systems) and these important structures are somehow better understood by these groups, see \cite[Chapter IV]{b:Beth-I}. These groups moreover arose as flag-transitive automorphism groups of $2$-designs in several contexts focusing on studying such designs with restricted conditions on automorphism groups and/or parameter sets, see \cite{a:AD-spor-ft,a:Regueiro-alt-spor,a:Zhou-lam3-spor,a:Zhou-sym-sporadic,a:ADM-PrimLam,a:Zhan-2017-lamcond-spor,a:Ding-2026-bprim-spor}. 
	This motivates us to investigate $2$-designs admitting flag-transitive almost simple automorphism groups $G$ with socle $G_{0}$ one of the  Mathieu groups in general. The case where $G$ is point-primitive and $G_{0}=\M_{11}$ or $\M_{12}$ have completely been determined \cite{a:Tian2020-M11,a:Tian2025-M12}, see also \cite{a:Crnkovic-M11,t:Wirth,a:Lempken-M12-imp}. Therefore, to achieve a complete classification, we revisit these two cases and study the possible point-imprimitive automorphism groups of such $2$-designs. In addition, we investigate the remaining possibilities, where the socle $G_{0}$ is one of the  Mathieu groups $\M_{22}$, $\M_{23}$ and $\M_{24}$, and prove the following result: 
	
	\begin{theorem}\label{thm:main}
		Let $\Dmc$ be a nontrivial $2$-$(v,k,\lambda)$ design admitting a flag-transitive automorphism group $G$ with socle $G_{0}$ one of the Mathieu groups. Then
		\begin{enumerate}[\rm (a)]
			\item if $G_{0}=\M_{11}$ or $\M_{12}$ and $G$ is point-imprimitive, then $G=\M_{12}$ and $\Dmc$ is the unique symmetric design with parameter set $(144, 66, 30)$ as in {\rm \cite{a:Lempken-M12-imp}};
			\item if $G_{0}=\M_{22}$, $\M_{23}$ or $\M_{24}$, then $G$ is point-primitive and $\Dmc$ is (isomorphic to) one of the $2$-designs listed in {\rm Table~\ref{tbl:des}}.
		\end{enumerate} 
	\end{theorem} 
	
	As noted above, the $2$-designs admitting flag-transitive and point-primitive almost simple automorphism groups with socle $G_{0}=\M_{11}$ or $\M_{12}$ have already been studied. This together with Theorem \ref{thm:main} presents a complete classification of $2$-designs admitting flag-transitive automorphism groups with socle one of the Mathieu groups, and in conclusion the  number of such designs is given in Table~ \ref{tbl:statistics}. It can obviously be seen that $\Dmc$ admits point-primitive automorphism groups except for the one where $G=\M_{12}$ and $\Dmc$ is the unique symmetric design on $144$ points with block size $66$ which, to our knowledge, is first appeared in \cite{a:Lempken-M12-imp} and reconstructed in \cite{a:AD-spor-ft}.
	
	\begin{corollary}\label{cor:main}
		Let $\Dmc$ be a nontrivial $2$-$(v,k,\lambda)$ design admitting a flag-transitive automorphism group $G$ with socle one of the Mathieu groups. Then one of the  following holds:
		\begin{enumerate}[\rm (a)]
			\item $G$ is point-imprimitive and $\Dmc$ is the unique symmetric design with parameter set $(144, 66, 30)$ as in {\rm \cite{a:Lempken-M12-imp}};
			\item $G$ is point-primitive and $\Dmc$ is known ({\rm Table~\ref{tbl:des}} apart from the design in line $26$).
		\end{enumerate}
	\end{corollary} 
	
	In order to prove Theorem \ref{thm:main}, we apply the method described in Section \ref{sec:method}. The modified versions of this method have been previously used in several papers of the second author and experts in the field, and in general, it can be applied to construct and to classify $2$-designs with flag-transitive automorphism groups. The proof relies on fundamental information of subgroup structures of the Mathieu groups and their permutation representations which can be read off from \cite{a:Leemans2015-atlas,b:Atlas}. We use the  computer software GAP \cite{GAP4} for computation. In particular, we use software packages ``\verb|AtlasRep|'' and ``\verb|Design|'' in GAP. 
	
	The detailed information on the $2$-designs obtained in Theorem \ref{thm:main} and Corollary~\ref{cor:main} is provided in Section \ref{sec:examples}. For the sake of completeness, in Table~\ref{tbl:des}, we list all $2$-designs admitting almost simple groups with socle one of the Mathieu groups. A base block of each of these designs is given in Table~\ref{tbl:base}. We note here that the known designs in the literature, for example \cite{a:Tian2020-M11,a:Tian2025-M12,a:Zhan-2017-lamcond-spor,a:Ding-2026-bprim-spor}, may be introduced with different base blocks. In this paper, to our knowledge, we obtain $101$ newly constructed flag-transitive non-isomorphic $2$-designs with block-imprimitive automorphism groups. 
	
	\setcounter{rownumber}{0}
	\addtolength{\tabcolsep}{-1.5pt} 
	\begin{table} 
		\caption{The number of non-trivial $2$-designs admitting flag-transitive automorphism groups $G$ with socle one of the Mathieu groups.} \label{tbl:statistics}
		\centering
		\scriptsize
		\begin{tabular}{lccccccccccc}
			\noalign{\smallskip}\hline\noalign{\smallskip}
			$G$ 
			& & Total 
			& & symmetric & non-symmetric 
			& & point-primitive & point-imprimitive 
			& & block-primitive & block-imprimitive\\
			\noalign{\smallskip}\cline{1-1}\cline{3-3}\cline{5-6}\cline{8-9}\cline{11-12}\noalign{\smallskip}
			
			$\M_{11}$ & & $16$ 
			& & $0$ & $16$ 
			& & $16$ & $0$ 
			& & $8$ & $8$ \\
			$\M_{12}$ & & $10$ 
			& & $2$ & $8$ 
			& & $9$ & $1$ 
			& & $6$ & $4$ \\
			$\M_{12}{:}2$ & & $2$ 
			& & $2$ & $0$ 
			& & $2$ & $0$ 
			& & $2$ & $0$ \\
			$\M_{22}$ & & $50$ 
			& & $1$ & $49$ 
			& & $50$ & $0$ 
			& & $11$ & $39$ \\
			$\M_{22}{:}2$ & & $35$ 
			& & $0$ & $35$ 
			& & $35$ & $0$ 
			& & $7$ & $28$  \\
			$\M_{23}$ & & $29$ 
			& & $0$ & $29$ 
			& & $29$ & $0$ 
			& & $9$ & $20$ \\
			$\M_{24}$ & & $42$ 
			& & $0$ & $42$ 
			& & $42$ & $0$ 
			& & $5$ & $37$ \\
			\noalign{\smallskip}\hline\noalign{\smallskip} 
		\end{tabular}
	\end{table}
	\addtolength{\tabcolsep}{1.5pt} 
	
	\subsection{Definitions and notation}\label{sec:defn}	
	All groups and incidence structures in this paper are finite.  Symmetric and alternating groups on $n$ letters are denoted by $\S_{n}$ and $\A_{n}$, respectively. We write ``$n$'' for the cyclic group of order $n$. We use ``$[n]$'' to denote a group of order $n$. A finite simple group is (isomorphic to) a cyclic group of prime order, an alternating group $\A_{n}$ for $n\geq 5$, a simple group of Lie type or a sporadic simple group, see \cite{b:Atlas} or \cite[Tables~5.1A-C]{b:KL-90}, and we use the standard notation as in \cite{b:Atlas} for finite simple groups and their subgroups. A $2$-$(v,k,\lambda)$ design $\Dmc$ is an incidence structure $(\Pmc,\Bmc)$ with a set $\Pmc$ of $v$ points and a set $\Bmc$ of blocks such that each block is a $k$-subset of $\Pmc$ and each pair of distinct points is contained in $\lambda$ blocks. We say $\Dmc$ is nontrivial if $2 < k < v-1$, and symmetric if $v = b$, where $b$ is the number of blocks of $\Dmc$.  A design $\Dmc$ is said to be \emph{complete} (or full) if $b=\binom{v}{k}$.  Each point of $\Dmc$ is contained in exactly $r=bk/v$ blocks which is called the \emph{replication number} of $\Dmc$. An \emph{automorphism} of a $2$-design $\Dmc$ is a permutation of the points permuting the blocks. The full automorphism group $\Aut(\Dmc)$ of $\Dmc$ is the group consisting of all automorphisms of $\Dmc$. A \emph{flag} of $\Dmc$ is a point-block pair $(\alpha, B)$ such that $\alpha \in B$. For $G\leq \Aut(\Dmc)$, the group $G$ is called \emph{flag-transitive} if $G$ acts transitively on the set of flags. The group $G$ is said to be \emph{point-primitive} or \emph{block-primitive} if $G$ acts primitively on $\Pmc$ or $\Bmc$, respectively.  
	Further notation and definitions in both design theory and group theory are standard and can be found, for example, in \cite{b:Atlas,b:Beth-I}.

	\section{Methodology}\label{sec:method}
	
	In this section, we describe the method that we use in this paper to prove Theorem \ref{thm:main}. We recall here that we use GAP \cite{GAP4}, in particular, we use stored groups library and software packages ``\verb|AtlasRep|'' and ``\verb|Design|''. 
	
	Suppose now that $\Dmc = (\Pmc, \Bmc)$ is a nontrivial non-symmetric $2$-$(v,k,\lambda)$ design. 
	Then 
	by Fisher's inequality \cite{b:Beth-I}, we  have that $v<b$ or $k<r$.
	Let $G$ be a flag-transitive automorphism group of  $\Dmc$, and let $(\alpha,B)$ be a flag of $\Dmc$. Then we consider the point-stabiliser $H:=G_{\alpha}$ and  the block-stabiliser $K:=G_{B}$. In what follows, we frequently use the following facts: 
	\begin{enumerate}[\rm \quad (i)]
		\item $r(k-1)=\lambda(v-1)$;
		\item $rv=bk$;
		\item $\lambda v<r^2$;
		\item $v=|G:H|$, $b=|G:K|$, $r=|H:L|$ and $k=|K:L|$, where $L:=G_{\alpha,B}=H_{B}=K_{\alpha}$;
		\item $r\mid |H|$ and $|G|<|H|^{3}$;
		\item $r\mid \lambda e$, for all nontrivial subdegrees $e$ of $G$;
		\item $K$ is transitive on $B$, and so $B$ is a $K$-orbit.  
	\end{enumerate} 
	We note here that the properties (iv)-(vii) follow from the fact that $G$ is flag-transitive, see for example \cite{a:Davies-87} and \cite[Lemma~2.1]{a:A-Exp-lam2}.
	
	Let now $G_{0}$ be one of the Mathieu groups, and let $G$ be an almost simple group with socle $G_{0}$. We can obtain all the subgroups of $G$ up to conjugation by GAP \cite{GAP4}. The list of these subgroups can also be read off from  \cite{a:Leemans2015-atlas}. For each group, this gives the possibilities for $H$ and $K$, and so by inspecting the subgroups of $G$ and using (i)-(v), we can find the possible tuples $[ (v,b,r,k,\lambda) , H, K ]$ with noting that $r = z/v$ and $k=z/b$, where $z$ is the index of a subgroup of $G$. 
	Now, for each tuple $ [ (v,b,r,k,\lambda) , H, K ]$, we consider the permutation representation of $G$ on the set of right cosets of $H$ in $G$, and view $G$ as the image of this representation. We first 
	check the subdegree property (vi), and for those possibilities satisfying this condition, we view $H$ and $K$ as images of this permutation representation (on the set of right cosets of $H$ in $G$). 
	This guarantees both $H$ and $K$ as permutation groups on $v$ points. We then check if property (vii) hold. If all the necessary conditions (i)-(vii) hold for a tuple $[ (v,b,r,k,\lambda) , H, K ]$, and we obtain a $K$-orbit $B$, then we check if $|B^{G}|=b$, and if this extra condition holds, we use the software package ``\verb*|Design|'' in GAP \cite{GAP4} to see if there exists a design with this parameter set. 
	In the case where $G$ is imprimitive on $v$ points and leaves invariant a nontrivial point-partition $\Cmc$ into $d$ parts of size $c$, we can employ \cite[Theorem 1 and Lemma 5]{a:DP-2021-Imp} to find the parameters $c$ and $d$. Indeed, the nonempty intersections $B\cap \Delta$ have a constant size $\ell$ (for $B\in \Bmc$ and $\Delta \in \Cmc$), and the integer $x=k-1-d(\ell-1)$ is positive, and the parameters $c$, $d$, $\ell$ and $x$ satisfy the inequalities and divisibility conditions listed in \cite[Lemma 5]{a:DP-2021-Imp}. In particular, $k\leq 2{\lambda}^2(\lambda-1)$ and $v\leq{(2{\lambda}^2(\lambda-1)-2)}^2$. We note here that some possibilities can be eliminated in two or more ways. For example, some point-imprimitive cases can also be eliminated by subdegree argument. In the proof of our statements in Section \ref{sec:proof}, we prefer to rule out these cases by taking imprimitive actions.    
	
	\section{Proof of the main result}\label{sec:proof}
	
	In this section, we prove Theorem \ref{thm:main}. In order to prove this result, we frequently use the method described in Section \ref{sec:method}, and we consider each group separately. 
	
	\begin{proposition}\label{prop:m11}
		Let $\Dmc$ be a non-trivial $2$-$(v,k,\lambda)$ design admitting a flag-transitive automorphism group $G$ with socle $G_{0}=\M_{11}$. Then $G$ is point-primitive and $(\Dmc,G)$ is as in one of the lines $1$-$16$ of {\rm Table \ref{tbl:des}}.
	\end{proposition}
	\begin{proof}
		Note that $G=G_{0}$ as $|\Out(G_{0})|=1$. We first find the subgroups of $G$ by GAP \cite{GAP4}, and in conclusion, $G$ has $39$ non-conjugate subgroups, five of which are maximal in $G$. 
		By \cite{a:AD-spor-ft}, we obtain no flag-transitive symmetric design. Therefore, assuming $\Dmc$ is non-symmetric, 
		we obtain $61$ parameters $(v,b,r,k,\lambda)$, and taking into account of the possibilities for candidates of point-stabiliser $H$, we have $61$ possibilities for $[(v, b, r, k, \lambda), H]$ in which $H$ is one of the  subgroups
		$3^2{:}\QD_{16}=\M_{9}{:}2$, $\A_{6}$, $\M_{10}=\A_{6}\cdot2$ and $\PSL_{2}(11)$. Among these possibilities, there are $6$ cases with non-maximal point-stabiliser in $G$ which may lead to point-imprimitive designs, and the remaining $55$ cases have maximal point-stabiliser.
		
		We first deal with the case where $G$ is point-imprimitive with a nontrivial point-partition $\Cmc$ into $d$ parts of size $c$ and $|B\cap \Delta|=\ell$ for some $B\in \Bmc$ and $\Delta \in \Cmc$. In this case, $H=\A_{6}$ has index $v=22$ in $G$, and we have one of the following parameter sets: 
		\begin{align*}
			&(22 , 66 , 45 , 15 , 30),  
			(22 , 132 , 90 , 15 , 60),
			(22 , 165 , 60 , 8 , 20), \\ 
			&(22 , 330 , 120 , 8 , 40), 
			(22 , 495 , 180 , 8 , 60),
			(22 , 990 , 360 , 8 , 120).	
		\end{align*}
		But none of these parameter sets leads to parameters $(c,d,\ell)$ satisfying the inequalities and divisibility conditions in \cite[Lemma 5]{a:DP-2021-Imp}. Therefore, $G$ is a point-primitive automorphism group of $\Dmc$, and hence the result follows from \cite[Theorem~1.1]{a:Tian2020-M11}.  
	\end{proof}
	
	\begin{proposition}\label{prop:m12}
		Let $\Dmc$ be a non-trivial $2$-$(v,k,\lambda)$ design with flag-transitive automorphism group $G$ with socle $G_{0}=\M_{12}$. Then $\Dmc$ is as in one of the lines $17$-$28$ of {\rm Table \ref{tbl:des}}.
	\end{proposition}
	\begin{proof}
		
		The case where $\Dmc$ is symmetric has been studied in \cite{a:AD-spor-ft,a:Zhou-sym-sporadic}, and hence we obtain one of the designs in line $25$ or $26$. The design in line $25$ is point-primitive and the  design in line $26$ is point-imprimitive. We now assume that $\Dmc$ is non-symmetric. Since $|\Out(G_0)|=2$, we have $G=G_0=\M_{12}$ or $G=\M_{12}{:}2$. In what follows, we discuss each case separately.
		
		We first suppose that $G=\M_{12}$. By GAP \cite{GAP4}, the group $G$ has $147$ non-conjugate subgroups, $11$ of which are maximal in $G$.
		Using the facts (i)-(v) in Section \ref{sec:method} and the method described there, we obtain $37$ parameter sets $(v, b, r, k, \lambda )$.
		Then by considering  the candidates for point-stabiliser $H$, we obtain $74$ possibilities for $[(v, b, r, k, \lambda), H]$ where $H$ is one of the subgroups 
		$\M_{10}{:}2=\A_{6}{:}2^{2}$ (two non-conjugate classes), $\M_{11}$ (two non-conjugate classes) and $\PSL_{2}(11)$ (two non-conjugate classes). Among these possibilities there are $68$ tuples with $H$ maximal in $G$ and $6$ tuples with $H$ non-maximal in $G$.  In the former case where $H$ is a maximal subgroup of $G$, the group $G$ is point-primitive, and all these possibilities have been treated in \cite{a:Tian2025-M12}, and in conclusion, we obtain one of the designs recorded as in lines $17$-$26$ of Table \ref{tbl:des}.

		In the latter case where $G$ is point-imprimitive, $H=\PSL_2(11)$ is contained in two non-conjugate maximal subgroups being isomorphic to $\M_{11}$. If $G$ leaves a nontrivial point-partition $\Cmc$ into $d$ parts of size $c$ and $|B\cap \Delta|=\ell$ for some $B\in \Bmc$ and $\Delta \in \Cmc$, then by \cite[Lemma 5]{a:DP-2021-Imp}, we obtain 
		the following possibilities for $[(v, b, r, k, \lambda), c, d, \ell]$:
		\begin{align*}
			&[ ( 144, 396, 110, 40, 30 ), 12, 12, 4],
			[ (144, 792, 220, 40, 60) , 12, 12, 4],\\
			&[ (144, 880, 165, 27, 30), 12, 12, 3],
			[ (144, 1760, 330, 27, 60), 12, 12, 3],\\
			&[ (144, 2376, 660, 40, 180), 12, 12, 4],
			[ (144, 3520, 660, 27, 120), 12, 12, 3].
		\end{align*}
		The candidates for $K$ of all these possibilities are recorded in Tables~\ref{tbl:m12-imp-norb} and \ref{tbl:m12-imp-orb}. 
		\setcounter{rownumber}{0}
		\begin{table}
			\caption{The possible parameters for $G=\M_{12}$ and non-maximal point-stabiliser $H=\PSL_2(11)$.} \label{tbl:m12-imp-norb}
			\centering
			\scriptsize
			\begin{tabular}{llllllllll}
				\noalign{\smallskip}\hline\noalign{\smallskip} 
				Line & $v$ & $b$ & $r$ & $k$ & $\lambda$ & $K$ & nr(K) & $K$-orbit   structure  \\
				\noalign{\smallskip}\hline\noalign{\smallskip} 
				
				\rownumber & $144$ & $396$ & $110$ & $40$ & $30$ & 
				$2{\times} \S_5$ & $18$ & $24^{1} 120^{1}$ &  \\ 
				\rownumber & $144$ & $792$ & $220$ & $40$ & $60$ & 
				$\S_5$ & $25$ & $2^{1} 10^{2} 12^{1} 20^{1} 30^{1} 60^{1}$ &  \\
				\rownumber & $144$ & $792$ & $220$ & $40$ & $60$ & 
				$\S_5$ & $26$ & $2^{1} 10^{2} 12^{1} 20^{1} 30^{1} 60^{1}$ &  \\
				\rownumber & $144$ & $792$ & $220$ & $40$ & $60$ & 
				$\S_5$ & $27$ & $12^{2} 60^{2}$ &  \\
				\rownumber & $144$ & $792$ & $220$ & $40$ & $60$ & 
				$\S_5$ & $28$ & $12^{2} 60^{2}$ &  \\
				\rownumber & $144$ & $792$ & $220$ & $40$ & $60$ & 
				$2{\times} \A_5$ & $29$ & $12^{2} 60^{2}$ &  \\
				\noalign{\smallskip}\hline\noalign{\smallskip} 
			\end{tabular}
		\end{table}
		If $K$ is one of the subgroups in Table~\ref{tbl:m12-imp-norb}, by GAP \cite{GAP4}, we obtain the $K$-orbit structure as in the last column of this table, and in conclusion, we observe that	$K$ has no orbit of length $k=|B|$, which is a contradiction. 
		For the remaining possibilities listed in Table~\ref{tbl:m12-imp-orb}, the subgroup $K$ has an orbit of length $k$. Such orbits are recorded in the last column of that table. Since $G$ is flag-transitive, the $G$-orbit $B^{G}$, where $B$ is a $K$-orbit as in Table~\ref{tbl:m12-imp-orb}, should form a block set of $\Dmc$, but using  GAP package ``\verb*|Design|'' \cite{DESIGN}, none of these orbits gives rise to a $2$-design.  
		
		Suppose now $G=\M_{12}{:}2$. By applying the  method described in Section \ref{sec:method}, we obtain $12$ possible parameter sets $(v, b, r, k, \lambda)$, and by taking into account of the candidates for the point-stabiliser
		$H$, we have $22$ possibilities for $[(v, b, r, k, \lambda), H]$, in all of which, $H$ is the maximal subgroup $(2^2{\times}\A_{5}){:}2$ or $\PSL_{2}(11){:}2$ (two non-conjugate classes), that is to say, $G$ is point-primitive. Then by  {\rm \cite[Theorem 1]{a:Tian2025-M12}} and \cite{a:Zhou-sym-sporadic}, we obtain one of the  possibilities recorded in lines $27$ and $28$ of  Table~\ref{tbl:des}. 
	\end{proof}
	
	\addtolength{\tabcolsep}{-2pt} 
	\setcounter{rownumber}{0}
	\begin{table}[h]
		\centering
		\scriptsize
		\caption{ The possible parameters for $G=\M_{12}$ and non-maximal point-stabiliser $H=\PSL_2(11)$.} \label{tbl:m12-imp-orb}
		\begin{tabular}{lllllllllp{75mm}}
			\noalign{\smallskip}\hline\noalign{\smallskip}
			Line & 
			$v$ & 
			$b$ & 
			$r$ & 
			$k$ & 
			$\lambda$ & 
			$K$ & 
			nr(K) &
			Orbit structure & 
			$K$-orbit \\
			\noalign{\smallskip}\hline\noalign{\smallskip}
			
			\rownumber & $144$ & $880$ & $165$ & $27$ & $30$ & 
			$3^{2+1}{:}2^{2}$ & $30$ & $9^{3} 18^{2} 27^{1} 54^{1}$ & $\{  4$, 8, 10, 18, 21, 22, 39, 43, 48, 49, 56, 58, 65, 68, 70, 74, 76, 80, 85, 90, 94, 125, 126, 128, 138, 141, $142  \}$  \\ 
			\rownumber & $144$ & $1760$ & $330$ & $27$ & $60$ & 
			$3^{2} {:}6$ & $48$ & $9^{5} 18^{1} 27^{1} 54^{1}$ &$\{4$, 8, 10, 18, 21, 22, 39, 43, 48, 49, 56, 58, 65, 68, 70, 74, 76, 80, 85, 90, 94, 125, 126, 128, 138, 141, $142\} $  \\
			\rownumber & $144$ & $1760$ & $330$ & $27$ & $60$ & 
			$3^{2} {:}6$ & $49$ & $9^{5} 18^{1} 27^{1} 54^{1}$ & $\{3$, 6, 9, 14, 19, 23, 26, 29, 35, 40, 44, 48, 53, 57, 59, 74, 81, 83, 116, 119, 120, 122, 126, 132, 133, 135, $136\}$ \\
			\multirow{3}{*}{\rownumber} & \multirow{3}{*}{$144$} & \multirow{3}{*}{$1760$} & \multirow{3}{*}{$330$} & \multirow{3}{*}{$27$} & \multirow{3}{*}{$60$} & 
			\multirow{3}{*}{$3^{2} {:}6$} & \multirow{3}{*}{$50$} & \multirow{3}{*}{$9^{3} 18^{2} 27^{3}$} & $\{1$, 5, 12, 14, 20, 24, 41, 42, 44, 50, 51, 55, 61, 64, 72, 75, 77, 81, 87, 89, 93, 123, 127, 132, 133, 135, $137\}$ \\
			& & & & & & & & & $\{3$, 6, 11, 13, 15, 17, 40, 45, 46, 52, 53, 60, 67, 69, 71, 79, 83, 84, 91, 95, 96, 121, 122, 131, 134, 140, $144\}$ \\
			& & & & &  & & & &  $\{ 4$, 8, 10, 18, 21, 22, 39, 43, 48, 49, 56, 58, 65, 68, 70, 74, 76, 80, 85, 90, 94, 125, 126, 128, 138, 141, $142\}$ \\
			\multirow{2}{*}{\rownumber} & \multirow{2}{*}{$144$} & \multirow{2}{*}{$2376$} & \multirow{2}{*}{$660$} & \multirow{2}{*}{$40$} & \multirow{2}{*}{$180$} & 
			\multirow{2}{*}{$2{\times} (5{:}4)$} & \multirow{2}{*}{$56$} & \multirow{3}{*}{$4^{1} 20^{3} 40^{2}$} & $\{1$, 3, 4, 6, 15, 16, 21, 24, 29, 32, 33, 35, 37, 44, 45, 48, 61, 68, 69, 72, 76, 77, 81, 83, 88, 91, 92, 95, 109, 110, 111, 115, 121, 126, 131, 132, 135, 136, 141, $144 \}$ \\
			& & & & & & & & & $\{ 5$, 7, 8, 10, 14, 18, 20, 23, 25, 27, 28, 30, 41, 42, 43, 46, 62, 64, 65, 66, 73, 75, 80, 82, 87, 89, 90, 93, 112, 117, 119, 120, 122, 123, 125, 129, 133, 137, 138, $140 \}$ \\
			\multirow{3}{*}{\rownumber} & \multirow{3}{*}{$144$} & \multirow{3}{*}{$3520$} & \multirow{3}{*}{$660$} & \multirow{3}{*}{$27$} & \multirow{3}{*}{$120$} & 
			\multirow{3}{*}{$3^{2} {:}3$} & \multirow{3}{*}{$69$} & \multirow{3}{*}{$9^{7} 27^{3}$} & $\{1$, 5, 12, 14, 20, 24, 41, 42, 44, 50, 51, 55, 61, 64, 72, 75, 77, 81, 87, 89, 93, 123, 127, 132, 133, 135, $137\}$ \\
			& & & & & & & & & $\{3$, 6, 11, 13, 15, 17, 40, 45, 46, 52, 53, 60, 67, 69, 71, 79, 83, 84, 91, 95, 96, 121, 122, 131, 134, 140, $144\}$ \\
			& & & & & & & & & $\{4$, 8, 10, 18, 21, 22, 39, 43, 48, 49, 56, 58, 65, 68, 70, 74, 76, 80, 85, 90, 94, 125, 126, 128, 138, 141, $142\}$ \\
			\noalign{\smallskip}\hline\noalign{\smallskip}
		\end{tabular}	
	\end{table}	
	
	\addtolength{\tabcolsep}{2pt} 
	\begin{proposition}\label{prop:m22}
		Let $\Dmc$ be a non-trivial $2$-$(v,k,\lambda)$ design admitting flag-transitive automorphism group $G$ with socle $G_{0}=\M_{22}$. Then $G$ is point-primitive and $(\Dmc,G)$ is as in one of the lines $29$-$113$ of {\rm Table \ref{tbl:des}}. 
	\end{proposition}
	
	\begin{proof}
		The symmetric case has already been studied in \cite{a:AD-spor-ft,a:Zhou-sym-sporadic}, and in conclusion we obtain the design in line $57$ of Table \ref{tbl:des} which is both point-primitive and block-primitive. Suppose now that $\Dmc$ is non-symmetric.
		Note that $|\Out(G_0)|=2$, and therefore $G=G_0=\M_{22}$ or $G=\M_{22}:2$. We study the two cases separately in the following. 
		
		Let $G=\M_{22}$. The list of all the subgroups (up to conjugation) of $\M_{22}$  can be found by GAP \cite{GAP4}. By  the method described in Section \ref{sec:method}, we obtain $200$ possibilities for $(v, b, r, k, \lambda)$.
		By considering the candidates for point-stabiliser $H$, we obtain $234$ possibilities for $[(v, b, r, k, \lambda), H]$, where $H$ is one of the following subgroups
		$2^3{:}\PSL_{3}(2)$, $2^4{:}\A_{6}$, $2^4{:}\S_{5}$, $\M_{10}=\A_{6}.2$, $\A_{7}$ (two non-conjugate classes) and $\PSL_{3}(4)$.
		We observe that in each case, $H$ is a maximal subgroup of $G$, that is to say, $G$ is point-primitive. Now by computing all non-trivial subdegrees of $G$, we conclude that all of these possibilities satisfy the property (vi) in Section \ref{sec:method}. Then by considering the candidates for block-stabiliser $K$, we have $827$ possibilities for the tuples $[(v, b, r, k, \lambda), H, K]$.
		For all such possibilities, we obtain all $K$-orbits, and then we observe that $371$ cases fail to have a $K$-orbit of length $k$. Now we obtain  $1247$ possibilities
		of the form $[[(v, b, r, k, \lambda), H, K], B]$, where $B$ is a $K$-orbit of length $k$. Furthermore, for each case, as $G$ acts transitively on
		the block set $\Bmc$, we must have $b=|B^G|=|\mathcal{B}|$. However, $865$ tuples fail to satisfy this condition. For the remaining $382$ possibilities, we use the GAP package ``\verb*|Design|'' \cite{DESIGN} and observe that $306$ possibilities lead to no design. Therefore, we have $76$ possible tuples $[[(v, b, r, k, \lambda), H, K], B]$, each of which leads to a $2$-design by taking $\Bmc=B^{G}$ with noting that some of these designs may be isomorphic. We first observe that there are $21$ designs with unique parameter sets, and so we have $21$ non-isomorphic designs. We then check possible isomorphism between $55$ remaining designs by  GAP package ``\verb*|Design|'' \cite{DESIGN}, and in conclusion, we find $28$ non-isomorphic designs. Therefore, if $G=\M_{22}$, then we obtain $49$ non-symmetric $2$-designs (up to isomorphism) listed as in lines $29$-$78$ of Table \ref{tbl:des}.
		
		Let  now $G=\M_{22}{:}2$. By applying the method in Section \ref{sec:method} and using the subgroups of $G$ (up to conjugation),
		and by considering candidates for point-stabiliser  $H$ we have $255$ possibilities for $[(v, b, r, k, \lambda), H]$, which $H$ is one of the subgroups $2^{5}{:}\S_{5}$, $\A_{6}{\cdot}2^{2}$, $2^4{:}\A_{6}$, $2^4{:}\S_{6}$, $2{\times}(2^3{:}\PSL_{3}(2))$, $\A_{7}$ and $\PSL_{3}(4){:}2$. 
		There are $18$ cases with $H$ non-maximal which are listed in Table \ref{tbl:m22d2-imp}, and the remaining $237$ cases, $H$ is maximal in $G$. We first consider $18$ point-imprimitive cases with a nontrivial point-partition $\Cmc$ into $d$ parts of size $c$ and $|B\cap \Delta|=\ell$ for some $B\in \Bmc$ and $\Delta \in \Cmc$. But none of these parameter sets leads to parameters $(c,d,\ell)$ satisfying the inequalities and divisibility conditions in \cite[Lemma 5]{a:DP-2021-Imp}. Therefore, $G$ must be point-primitive, and we need to discuss $237$ remaining possibilities.  All of these possibilities satisfy the property (vi) in Section \ref{sec:method}. Now by considering candidates for block-stabiliser $K$, we obtain  $1579$ possibilities for $[(v, b, r, k, \lambda), H, K]$. For each case, we compute the $K$-orbits, and then we exclude $809$ tuples that have no $K$-orbit of length $k$. By appending the $K$-orbit $B$ of length $k$, we obtain  $1218$ tuples of the form $[[(v, b, r, k, \lambda), H, K], B]$. We then observe that $1046$  possibilities fail to satisfy the condition $|B^G|=b$. In conclusion, by using GAP package ``\verb*|Design|'' \cite{DESIGN}, we obtain $35$ non-symmetric designs up to isomorphism recorded as in one of the lines $79$-$113$ of Table \ref{tbl:des}. 
	\end{proof}

	\setcounter{rownumber}{0}
	\begin{table} 
		\caption{The possible parameters for $G=\M_{22}{:}2$ as point-imprimitive automorphism group of $\Dmc$} \label{tbl:m22d2-imp}
		\centering
		\scriptsize
		\begin{tabular}{llllllll}
			\noalign{\smallskip}\hline\noalign{\smallskip} 
			Line & $v$ & $b$ & $r$ & $k$ & $\lambda$ & $H$ & nr(H) \\
			\noalign{\smallskip}\hline\noalign{\smallskip} 
			
			\rownumber & $154$ & $231$ & $180$ & $120$ & $140$ & $2^4{:}\A_{6}$ & $6$\\
			\rownumber & $154$ & $462$ & $360$ & $120$ & $280$ & $2^4{:}\A_{6}$ & $6$\\
			\rownumber & $154$ & $616$ & $72$ & $18$ & $8$ & $2^4{:}\A_{6}$ & $6$\\
			\rownumber & $154$ & $770$ & $90$ & $18$ & $10$ & $2^4{:}\A_{6}$ & $6$\\
			\rownumber & $154$ & $924$ & $720$ & $120$ & $560$ & $2^4{:}\A_{6}$ & $6$\\
			\rownumber & $154$ & $1232$ & $144$ & $18$ & $16$ & $2^4{:}\A_{6}$ & $6$\\
			\rownumber & $154$ & $1540$ & $180$ & $18$ & $20$ & $2^4{:}\A_{6}$ & $6$\\
			\rownumber & $154$ & $2464$ & $288$ & $18$ & $32$ & $2^4{:}\A_{6}$ & $6$\\
			\rownumber & $154$ & $3080$ & $360$ & $18$ & $40$ & $2^4{:}\A_{6}$ & $6$\\
			\rownumber & $154$ & $6160$ & $720$ & $18$ & $80$ & $2^4{:}\A_{6}$ & $6$\\
			\rownumber & $154$ & $7392$ & $5760$ & $120$ & $4480$ & $2^4{:}\A_{6}$ & $6$\\
			\rownumber & $154$ & $12320$ & $1440$ & $18$ & $160$ & $2^4{:}\A_{6}$ & $6$\\
			\rownumber & $154$ & $24640$ & $2880$ & $18$ & $320$ & $2^4{:}\A_{6}$ & $6$\\
			\rownumber & $154$ & $49280$ & $5760$ & $18$ & $640$ & $2^4{:}\A_{6}$ & $6$\\
			\rownumber & $352$ & $2772$ & $315$ & $40$ & $35$ & $\A_{7}$ & $9$\\
			\rownumber & $352$ & $5544$ & $630$ & $40$ & $70$ & $\A_{7}$ & $9$\\
			\rownumber & $352$ & $11088$ & $1260$ & $40$ & $140$ & $\A_{7}$ & $9$\\
			\rownumber & $352$ & $22176$ & $2520$ & $40$ & $280$ & $\A_{7}$ & $9$\\
			\noalign{\smallskip}\hline\noalign{\smallskip} 
		\end{tabular}
	\end{table}

	\begin{proposition}\label{prop:m23}
		Let $\Dmc$ be a non-trivial $2$-$(v,k,\lambda)$ design with flag-transitive automorphism group $G$ with socle $G_{0}=\M_{23}$. Then $G$ is point-primitive and $(\Dmc , G)$ is as in one of the lines $114$-$142$ of {\rm Table \ref{tbl:des}}. 
	\end{proposition}
	
	\begin{proof}
		Since $\Out(G_0)=1$, we have $G=\M_{23}$. By \cite{a:AD-spor-ft,a:Zhou-sym-sporadic}, we have no symmetric design with flag-transitive automorphism group $G=\M_{23}$. Let now $\Dmc$ be a non-trivial non-symmetric $2$-$(v, k, \lambda)$ design. Then by applying the method described in Section \ref{sec:method}, we obtain $518$ parameter sets $(v, b, r, k, \lambda)$.
		By considering candidates for point-stabiliser $H$, we have $838$ possibilities for $[(v, b, r, k, \lambda), H]$, where $H$ is one of the subgroups $2^4{:}(3{\times}\A_{5}){:}2$, $2^4{:}\A_{6}$, $2^4{:}\A_{7}$, $2^4{:}\S_{5}$ (two non-conjugate classes), $\M_{11}$, $\M_{22}$ and $\PSL_{3}(4){:}2$.
		We observe that among these possibilities, there are $12$ cases with non-maximal subgroup $H$, which are listed in Table \ref{tbl:m23-imp}, and $826$ cases with maximal point-stabiliser $H$ in $G$.

		We first consider $12$ point-imprimitive cases with a nontrivial point-partition $\Cmc$ into $d$ parts of size $c$ and $|B\cap \Delta|=\ell$ for some $B\in \Bmc$ and $\Delta \in \Cmc$. But none of these parameter sets leads to parameters $(c,d,\ell)$ satisfying the inequalities and divisibility conditions in \cite[Lemma 5]{a:DP-2021-Imp}. Therefore, $G$ is point-primitive.

		Now we consider the remaining $826$ cases with maximal subgroup $H$ in $G$. Among these, only $236$ possibilities satisfy the subdegree property (vi) in Section \ref{sec:method}, and then by considering candidates for block-stabiliser $K$, we obtain $688$ possibilities for $[(v, b, r, k, \lambda), H, K]$. However, $431$ cases admit no $K$-orbit of length $k$. Thus we have $438$ possibilities for $[[(v, b, r, k, \lambda), H, K], B]$ with $K$-orbit $B$ of length $k$. Furthermore, $291$ of these possibilities do not satisfy the condition $|B^G|=b$, leaving $147$ possibilities for further consideration. In conclusion, using GAP package ``\verb*|Design|'' \cite{DESIGN}, we find exactly $29$ designs up to isomorphism, which are listed as in one of the lines $114$-$142$ of {\rm Table \ref{tbl:des}}.
	\end{proof}

	\setcounter{rownumber}{0}
	\begin{table}
		\caption{The possible parameters for $G=\M_{23}$ as point-imprimitive automorphism group of $\Dmc$} \label{tbl:m23-imp}
		\centering
		\scriptsize
		\begin{tabular}{llllllll}
			\noalign{\smallskip}\hline\noalign{\smallskip} 
			Line & $v$ & $b$ & $r$ & $k$ & $\lambda$ & $H$ & nr(H) \\
			\noalign{\smallskip}\hline\noalign{\smallskip} 
			
			\rownumber & $1771$ & $3542$ & $120$ & $60$ & $4$ & $2^4{:}\A_{6}$ & $9$\\
			\rownumber & $1771$ & $5313$ & $180$ & $60$ & $6$ & $2^4{:}\A_{6}$ & $9$\\
			\rownumber & $1771$ & $10626$ & $360$ & $60$ & $12$ & $2^4{:}\A_{6}$ & $9$\\
			\rownumber & $1771$ & $14168$ & $480$ & $60$ & $16$ & $2^4{:}\A_{6}$ & $9$\\
			\rownumber & $1771$ & $21252$ & $720$ & $60$ & $24$ & $2^4{:}\A_{6}$ & $9$\\
			\rownumber & $1771$ & $28336$ & $960$ & $60$ & $32$ & $2^4{:}\A_{6}$ & $9$\\
			\rownumber & $1771$ & $42504$ & $1440$ & $60$ & $48$ & $2^4{:}\A_{6}$ & $9$\\
			\rownumber & $1771$ & $56672$ & $1920$ & $60$ & $64$ & $2^4{:}\A_{6}$ & $9$\\
			\rownumber & $1771$ & $85008$ & $2880$ & $60$ & $96$ & $2^4{:}\A_{6}$ & $9$\\
			\rownumber & $1771$ & $170016$ & $5760$ & $60$ & $192$ & $2^4{:}\A_{6}$ & $9$\\
			\rownumber & $5313$ & $60720$ & $960$ & $84$ & $15$ & $2^4{:}\S_{5}$ & $15$\\
			\rownumber & $5313$ & $60720$ & $960$ & $84$ & $15$ & $2^4{:}\S_{5}$ & $16$\\
			\noalign{\smallskip}\hline\noalign{\smallskip} 
		\end{tabular}
	\end{table}

	\begin{proposition}\label{prop:m24}
		Let $\Dmc$ be a non-trivial $2$-$(v,k,\lambda)$ design with a  flag-transitive automorphism group $G$ with socle $G_{0}=\M_{24}$. Then $G$ is point-primitive and $(\Dmc , G)$ is as in one of the lines $143$-$184$ of {\rm Table \ref{tbl:des}}.
	\end{proposition}
	
	\begin{proof}
		Here $G=\M_{24}$ since  $|\Out(G_0)|=1$. It follows from  \cite{a:AD-spor-ft,a:Zhou-sym-sporadic} that $\Dmc$ must be non-symmetric. 
		Applying the method described in Section \ref{sec:method}, leads us to $497$ possibilities for $[(v, b, r, k, \lambda), H]$, where $H$ is one of the maximal subgroups $2^{6} {:} 3 {\cdot} \S_{6}$, $\M_{12} {:} 2$, $\M_{22} {:} 2$ and $\M_{23}$. Therefore, $G$ is point-primitive. The subdegree property (vi) rules out $147$ cases leaving $350$ cases for further consideration. Now by inspecting the subgroups $K$ as admissible block-stabilisers, we obtain $1760$ possibilities for $[(v, b, r, k, \lambda), H, K]$. Among these, for $1242$ cases, we find no $K$-orbit of length $k$, and hence these cases cannot occur. Therefore, we have $1319$
		possibilities for $[[(v, b, r, k, \lambda), H, K], B]$ with $K$-orbit $B$ of length $k$. There are $845$ cases not satisfying the condition $|B^G|=b$, and this leaves $474$ possibilities. Moreover by GAP package ``\verb*|Design|'' \cite{DESIGN}, we obtain $42$ designs up to isomorphism, and these $2$-designs  are listed as in one of the lines $143$-$184$ of Table \ref{tbl:des}.	
	\end{proof}
	
	\noindent \textbf{Proof of Theorem \ref{thm:main}.} 
	Let $\Dmc$ be a nontrivial $2$-$(v,k,\lambda)$ design admitting a flag-transitive automorphism group $G$ with socle $G_{0}$ one of the Mathieu groups. If $G_{0}=\M_{11}$ or $\M_{12}$, then by 
	\cite{a:Tian2020-M11,a:Tian2025-M12}, we only need to consider the point-imprimitive automorphism groups $G$, and this case is treated in Propositions~\ref{prop:m11}-\ref{prop:m12}, and we conclude that $G=\M_{12}$ and $\Dmc$ is the unique symmetric design with parameter set $(144, 66, 30)$ as in \cite{a:Lempken-M12-imp}, see also \cite[Theorem 1.1(a)]{a:AD-spor-ft}. In the remaining cases where $G_{0}=\M_{22}$, $\M_{23}$ or $\M_{24}$, the result follows immediately from Propositions~\ref{prop:m22}-\ref{prop:m24}.

	\section{Examples: construction of designs}\label{sec:examples}
	
	In this section, we give brief information on the construction of the  $2$-designs obtained in Theorem \ref{thm:main} (and Corollary~\ref{cor:main}). We present some information on the symmetries of $2$-designs as well as required information for their construction. Indeed, each line of Table~\ref{tbl:des} represents a $2$-design with the parameters in the line. We now fix a design $\Dmc$ in a line.  
	The second column is the number associated with $\Dmc$ obtained from the automorphism  group $G$ in column $8$. The parameters of $\Dmc$ are given in columns $3$-$7$, and the group $G$ is written in the $8$th column. The structures of point-stabiliser $H$ and block-stabiliser $K$ are given in columns $9$ and $10$, respectively. The numbers $\nr(H)$ and $\nr(K)$ associated with these subgroups are listed in the next two columns, namely, $11$ and $12$. These numbers in particular indicate whether the designs are constructed on the same point set or not. In column $13$, we write the full automorphism group $\Aut(\Dmc)$ of the design, and we indicate in the next two columns $14$ and $15$ if $G$ is point-primitive ``p-prim'' and/or block-primitive ``b-prim''. 
	If there is a design with number $i$ which is the complement of $\Dmc$,  then we write $i$ in column $16$ with the header ``comp'', otherwise, we leave it blank. Any comment regarding $\Dmc$ is recorded in column $17$. In the last column, we present some references in which a construction of the design $\Dmc$ is given. 
	
	Suppose now that $\Dmc_i(G)$ is the design with number $i$ in the second column and $G$ in the $8$th column. In order to construct the design $\Dmc_i(G)$, we list a base block $\Bb{G}{i}$ of $\Dmc_{i}(G)$ in Table~\ref{tbl:base}. Then, as explained in Section \ref{sec:method}, the point set of $\Dmc_i(G)$ is $\{1,\ldots,v\}$ and the block set of $\Dmc_i(G)$ is the $G$-orbit $\Bb{G}{i}^{G}$. 
	For the design $\Dmc_{10}(\M_{12})$ which is point-imprimitive, an explicit construction is given in \cite{a:AD-spor-ft}. The generators of $G$, $H$ and $K$ are given in \cite[Table~2]{a:AD-spor-ft}, and we take  $\Bb{\M_{12}}{10}= \{1,$ 3, 4, 5, 8, 11, 13, 14, 16, 19, 20, 22, 26, 27, 28, 32, 33, 36, 37, 38, 41, 42, 44, 48, 49, 50, 51, 58, 59, 60, 64, 66, 68, 70, 71, 72, 73, 74, 75, 82, 83, 84, 98, 102, 103, 105, 106, 108, 109, 110, 112, 115, 116, 118, 121, 123, 126, 128, 129, 130, 135, 136, 139, 141, 142, $143\}$. Then the  block set of this design is $\Bb{\M_{12}}{10}^{G}$. For the  remaining designs, as all $2$-designs are point-primitive, one can use the GAP library for primitive permutation groups and take $G$ as a primitive group on $v$ points, and then follow the  above argument to construct the design. Note in passing that in our original arguments in Propositions~\ref{prop:m11}-\ref{prop:m24}, we take $\Pmc$ as the  set of right cosets of $H$ in $G$, and this  leads to different generators for $G$, $H$ and $K$ in comparison with the ones stored in GAP's libraries. But for convenience, in Table~7, we present (a new) base block by using available data in GAP's library \cite{GAP4}. We remark here that for $G=\M_{12}:2$ we have two primitive permutation groups on $144$ points stored in GAP, the designs in lines $27$ and $28$ are obtained by taking $G$ as the $5$th group and the $4$th group in the library, respectively. Moreover, $G=\M_{23}$, as the $5$th primitive group in the GAP library gives rise to the designs in lines $138$-$142$. 
	
	\section*{Declaration of competing interest}
	
	The authors declare that they have no known competing financial interests or personal relationships that could have appeared to influence the work reported in this paper.
	
	\section*{Data availability statement}
	
	The authors declare that the data supporting the findings of this study are available within the paper and its references.

	\setcounter{rownumber}{0}
	\setcounter{linenumber}{0}
	\tiny
	\addtolength{\tabcolsep}{-4pt} 
\begin{longtable}{@{}llllllllllllllllll@{}}
		\caption{Nontrivial flag-transitive $2$-designs $\Dmc$ with  Mathieu groups as automorphism groups}\label{tbl:des} \\
		\noalign{\smallskip}\hline\noalign{\smallskip}
		Line & 
		Nr & $v$ & $b$ & $r$ & $k$ & $\lambda$ & 
		$G$ & $H$ & $K$ & nr($H$) & nr($K$) & $\Aut(\Dmc)$ & p-prim & b-prim & comp & comment  & References \\
		\noalign{\smallskip}\hline\noalign{\smallskip}
		\endfirsthead
		\multicolumn{11}{c}{\tablename\ \thetable{} -- continued} \\
		\noalign{\smallskip}\hline\noalign{\smallskip}
		
		Line & 
		Nr & $v$ & $b$ & $r$ & $k$ & $\lambda$ & $G$ & $H$ & $K$ & nr($H$) & nr($K$) & $\Aut(\Dmc)$ & p-prim & b-prim & comp & comment  & References \\
		\noalign{\smallskip}\hline\noalign{\smallskip}
		\endhead
		\noalign{\smallskip}\hline
		\multicolumn{18}{r}{ \raggedright-- continued} \\
		\endfoot
		\noalign{\smallskip}\hline
		\endlastfoot
		
		\linenumber & \rownumber  & $11$ & $55$ & $45$ & $9$ & $36$ & $\M_{11}$ & $\M_{10}$ & $\M_{9}{:}2$ & $2$ & $5$ & $\S_{11}$ & true & true &   & complete & \cite{a:Ding-2026-bprim-spor,a:Crnkovic-M11}\\
		\linenumber & \rownumber  & $11$ & $66$ & $30$ & $5$ & $12$ & $\M_{11}$ & $\M_{10}$ & $\S_{5}$ & $2$ & $6$ & $\M_{11}$ & true & true & $3$ & & \cite{a:Tian2020-M11, a:Ding-2026-bprim-spor, a:smallk-5-spor,a:Crnkovic-M11} \\
		\linenumber & \rownumber  & $11$ & $66$ & $36$ & $6$ & $18$ & $\M_{11}$ & $\M_{10}$ & $\S_{5}$ & $2$ & $6$ & $\M_{11}$ & true & true & $2$& &\cite{a:Tian2020-M11, a:Ding-2026-bprim-spor,a:Crnkovic-M11}\\ 
		\linenumber & \rownumber  & $11$ & $165$ & $45$ & $3$ & $9$ & $\M_{11}$ & $\M_{10}$ & $\M_{8}{:}\S_{3}$ & $2$ & $13$ & $\S_{11}$ & true & true & $5$ & complete  & \cite{a:Tian2020-M11, a:Ding-2026-bprim-spor, a:Zhan2026-bt-spor-lam5,a:Crnkovic-M11}\\
		\linenumber & \rownumber  & $11$ & $165$ & $120$ & $8$ & $84$ & $\M_{11}$ & $\M_{10}$ & $\M_{8}{:}\S_{3}$ & $2$ & $13$ & $\S_{11}$ & true & true & $4$ & complete & \cite{a:Tian2020-M11, a:Ding-2026-bprim-spor,a:Crnkovic-M11}\\
		\linenumber & \rownumber  & $11$ & $330$ & $120$ & $4$ & $36$ & $\M_{11}$ & $\M_{10}$ & $\S_{4}$ & $2$ &  $17$ & $\S_{11}$ & true & false &    & complete & \cite{a:Tian2020-M11,a:Crnkovic-M11}\\
		\linenumber & \rownumber  & $11$ & $396$ & $180$ & $5$ & $72$ & $\M_{11}$ & $\M_{10}$ & $5{:}4$ & $2$ & $19$ & $\M_{11}$ & true & false &    &    & \cite{a:Tian2020-M11,a:Crnkovic-M11}  \\
		\linenumber & \rownumber  & $12$ & $22$ & $11$ & $6$ & $5$ & $\M_{11}$ & $\PSL_2(11)$ & $\A_{6}$ & $3$ & $4$ & $\M_{11}$ & true & false & $8$ & & \cite{a:Tian2020-M11, a:Zhan-CP-nonsym-sprodic, a:Zhan-2017-lamcond-spor, a:ADM-PrimLam, a:Zhan2026-bt-spor-lam5,a:Crnkovic-M11}\\
\end{longtable}
	\normalsize
	
	\scriptsize
	\addtolength{\tabcolsep}{1pt} 
	\begin{longtable}[h!]{lp{13cm}}
		\caption{Base blocks of flag-transitive $2$-designs with Mathieu groups as their automorphism groups} \label{tbl:base}\\
		\noalign{\smallskip}\hline\noalign{\smallskip}
		Name & Base block \\
		\noalign{\smallskip}\hline\noalign{\smallskip}
		\endfirsthead
		
		\multicolumn{2}{c}{\tablename\ \thetable{} -- continued} \\
		\noalign{\smallskip}\hline\noalign{\smallskip}
		Name & Base block \\
		\noalign{\smallskip}\hline\noalign{\smallskip}
		\endhead
		
		\noalign{\smallskip}\hline
		\endfoot
		
		\noalign{\smallskip}\hline
		\endlastfoot

		$\Bb{\M_{11}}{1}$ & $\{ 2$, 3, 4, 5, 6, 7, 9, 10, $11 \}$ \\
		$\Bb{\M_{11}}{2}$ & $\{ 2$, 4, 5, 9, $11 \}$ \\
		$\Bb{\M_{11}}{3}$ & $\{ 1$, 3, 6, 7, 8, $10 \}$	\\
		$\Bb{\M_{11}}{4}$ & $\{ 1$, 8, $9 \}$ \\
		$\Bb{\M_{11}}{5}$ & $\{ 2$, 3, 4, 5, 6, 7, 10, $11 \}$ \\
		$\Bb{\M_{11}}{6}$ & $\{ 4$, 5, 7, $10 \}$ \\
		$\Bb{\M_{11}}{7}$ & $\{ 3$, 5, 6, 9, $11 \}$ \\
		$\Bb{\M_{11}}{8}$ & $\{ 1$, 2, 3, 5, 6, $8 \}$ \\
\end{longtable}
	\normalsize
	

\end{document}